\documentclass[12pt]{amsart}
\usepackage{amsmath,amsfonts,amssymb} 

\usepackage{graphicx}
\usepackage{natbib} 
\usepackage{verbatim}
\usepackage{lineno}

\begin{document}

\bibliographystyle{plainnat}
\newcommand{\PS}{\mathcal{P}}
\newcommand{\id}{\mbox{Id}}
\newcommand{\supp}{\mbox{supp}}
\newcommand{\R}{\mathbf{R}}
\newcommand{\E}{\mathbb{E}}
\renewcommand{\P}{\mathbb{P}}
\newcommand{\bR}{\partial\mathbf{R}}
\newcommand{\Rk}{\R^k}  
\newcommand{\Rkp}{\Rk_+}
\renewcommand{\S}{\mathbf{S}}  
\newcommand{\iS}{\S\setminus \S_0}  
\newcommand{\Cp}{\I\Sn }  
\newcommand{\iDelta}{\Delta\setminus\Delta_0}
\def\cL{\mathcal{L}}
\def\cA{\mathcal{A}}
\def \e{\omega}

\newtheorem{theorem}{Theorem}[section]
\newtheorem{corollary}[theorem]{Corollary}
\newtheorem{question}{Question}
\newtheorem{conjecture}{Conjecture}
\newtheorem{lemma}[theorem]{Lemma}
\newtheorem{proposition}[theorem]{Proposition}
\newtheorem{definition}{Definition}
\newtheorem{remark}{Remark}
\newtheorem{hypothesis}[theorem]{Hypothesis}
\newtheorem{example}[theorem]{Example}
\newcommand{\nz}{\hfill\break}
\newcommand{\lin}{\mbox{lin }}

\title{Persistence for stochastic difference equations:\\ A mini-review}

\author{Sebastian J. Schreiber}
\address{Department of Evolution and Ecology and the Center for Population Biology\\
University of California, Davis, California 95616}
\email{sschreiber@ucdavis.edu}

\maketitle
\begin{abstract}
Understanding under what conditions populations, whether they be plants, animals, or viral particles, persist is an issue of theoretical and practical importance in population biology. Both biotic interactions and environmental fluctuations are key factors that can facilitate or disrupt persistence. One approach to examining the interplay between these deterministic and stochastic forces is the construction and analysis of stochastic difference equations $X_{t+1}=F(X_t,\xi_{t+1})$ where $X_t \in \R^k$ represents the state of the populations and $\xi_1,\xi_2,\dots$ is a sequence of random variables representing environmental stochasticity. In the analysis of these stochastic models, many theoretical population biologists are interested in whether the models are bounded and persistent. Here, boundedness asserts that asymptotically $X_t$ tends to remain in compact sets. In contrast, persistence requires that $X_t$ tends to be ``repelled'' by some ''extinction set'' $S_0\subset \R^k$. Here, results on both of these proprieties are reviewed for single species, multiple species, and structured population models. The results are illustrated with applications to stochastic versions of the Hassell and Ricker single species models, Ricker, Beverton-Holt,  lottery models of competition, and lottery models of rock-paper-scissor games. A variety of conjectures and suggestions for future research are presented. 
\end{abstract}

\linenumbers
\vskip 0.1in
\centerline{\large To appear in the \emph{Journal of Difference Equations and Applications}}

\section{Introduction}

One of the most fundamental equations in population biology is ``what are the necessary conditions to ensure the long-term persistence of a population or a collection of interacting populations?'' A fruitful approach to addressing this question has been the development and analysis of mathematical models. For deterministic models, such as difference or differential equations, any reasonable definition of persistence requires the existence of an attractor bounded away from extinction of one or more of the populations~\citep{jtb-06}. When this attractor is a global attractor (i.e. its basin includes all non-extinction states), the system is uniformly persistent or permanent~\citep{schuster-etal-79,hofbauer-81,butler-freedman-waltman-86}. Permanence ensures populations recover from infrequent large perturbations often experienced by biological systems~\citep{jansen-sigmund-98,jtb-06}. Since its introduction, methods for verifying permanence have been developed extensively for deterministic models accounting for nonlinear species interactions, stage-structure within populations, and spatial heterogeneity (see, e.g, the books of \citet{hofbauer-sigmund-98,cantrell-cosner-03,zhao-03,smith-thieme-11}).  

Temporal fluctuations in environmental conditions also play a crucial role in determining persistence. One approach to understanding the influence of these temporal fluctuations is the study of uniform persistence for non-autonomous difference or differential equations~\citep{thieme-00,mierczyski-etal-04,smith-thieme-11}. A strength of this approach is that it allows for relatively arbitrary forms of temporal fluctuations including periodic, quasi-periodic and stochastic motions. However, the definition of persistence in these studies often requires that the population trajectories remain uniformly bounded away from the extinction state. For many stochastic models where the vagaries of the environment are encapsulated in randomly varying parameters, this requirement is too strong~\citep{gillespie-73,chesson-warner-81,turelli-81,chesson-94,ellner-sasaki-96,bjornstad-grenfell-01,kuang-chesson-08,kuang-chesson-09}.  For these models, population trajectories often drift arbitrarily close to the extinction set. However, under certain conditions, there may be a probabilistic tendency to stay away from this extinction set~\citep{chesson-78,chesson-82}. 

Here, I provide a partial review of the latter approach to  persistence in fluctuating environments. The main class of stochastic difference equations under consideration are introduced in Section 2. Definitions of boundedness and persistence are given for stochastic difference equations. For the former, a rather general theorem is presented. For the latter, results are  more dependent on model structure. Consequently, Sections 3, 4, and 5 discuss results for scalar models, multispecies models, and structured single species models, respectively. Section 6 concludes with parting comments and suggestions for future research.

\section{Background}

To study population dynamics in a random environment, consider stochastic difference equations of the form
\begin{equation}\label{eq:background1}
X_{t+1}=F(X_t,\xi_{t+1})
\end{equation}
where $X_t\in \S$ represents the ``state'' of the population at time $t$ (e.g. a vector of densities or frequencies) and $\xi_{t}$ is a random variable that determines the ``environmental conditions'' at time $t$. Throughout this article, I make \emph{the following standing assumptions}. 
\begin{description}
\item [A1]  $\{\xi_t\}_{t=0}^\infty$ is a sequence of i.i.d random variables taking values in a separable metric space $\Omega$ (such as $\R^n$).
\item [A2] $F:\S\times \Omega \to \S$ is a continuous function where $\S$ is a closed subset of $\R^k$. 
\item [A3] There is a closed subset $\S_0\subset \S$ such that $\S_0$ and $\S\setminus \S_0$ are invariant i.e. $F(x,\omega)\in \S_0$ if and only if $x\in \S_0$ and $F(x,\omega)\in \S\setminus \S_0$ if and only if $x\in \S\setminus \S_0$.  
\end{description}
Assumptions \textbf{A1}--\textbf{A2} imply that $\{X_t\}_{t=0}^\infty$ is a Markov chain on $\S$ and that $\{X_t\}_{t=0}^\infty$ is Feller, meaning $x\mapsto \E[h(X_1)|X_0=x]$ is bounded and continuous whenever $h:\S\to\R$ is bounded and continuous. For the many of the results presented here, $\S$ is either the non-negative orthant $\R^k_+$ of $\R^k$ in which case $x \in \S$ is a vector of population densities or $\S$ is the probability simplex $\Delta=\{x\in \R^k_+: \sum_i x_i=1\}$ in which case $x\in \S$ is a vector of population frequencies. $\S_0$ in assumption \textbf{A3} is interpreted as the ``extinction set''  where one or more populations have gone extinct. The invariance of $\S_0$ implies that once the population has gone extinct it remains extinct i.e. the ``no cats, no kittens'' principle in population biology. Alternatively, the invariance of $\S\setminus \S_0$ implies that populations can not go extinct in one time step but only asymptotically, an assumption met by most of the models in the population models in the literature. In particular, these models do not account for demographic stochasticity which stems from the finiteness of populations. 

It is natural to study the asymptotic behavior of \eqref{eq:background1} from two perspectives. First, one might ask ``what is the probability the populations are in a particular configuration far into the future?'' More precisely, given a Borel set $B\subset \S$, what can we say about $\P[X_t\in B]$ (i..e. the probability $X_t \in B$) for large $t$? Since these probabilities correspond to the frequency of observing a particular event across many realizations of the stochastic process, answering this question provides  an ``ensemble view'' of the long-term dynamics.  An alternative perspective corresponds to asking ``how frequently does the typical population trajectory  visit a particular configuration far into the future?'' To answer this question, it useful to introduce the \emph{empirical measures} for the Markov chain $X_t$ given by
\[
\Pi_t = \frac{1}{t}\sum_{s=0}^{t-1} \delta_{X_s}
\]
where $\delta_x$ denotes a Dirac measure at the point $x$ i.e. $\delta_x(A)=1$ if $x\in A$ and $0$ otherwise. For any Borel set $A\subset \S$, $\Pi_t(A)$ is the fraction of time that $X_s$ spends in $A$ for $1\le s\le t$. Provided the limit exists, the long-term frequency  that $X$ enters $A$ is given by $\lim_{t\to\infty} \Pi_t(A)$. Understanding this asymptotic behavior with probability one corresponds to the ``typical trajectory'' perspective.

 This review focuses on two aspects of the asymptotic behavior: boundedness in which the populations tend to stay bounded and persistence in which the populations tend to stay away from the extinction set $\S_0$. Both aspects are viewed from the ``ensemble'' and ``typical trajectory'' perspectives.

\subsection{Boundedness}

When studying models of population dynamics, the first question that  comes to mind is ``are the long term population abundances bounded?'' After all, we live in a finite world so population numbers can not become arbitrary large for indefinitely long periods of time. For deterministic models, an appropriate notion of boundedness is \emph{dissipativeness}: the existence of a compact set $C\subset \S$ such that all solutions of \eqref{eq:background1} eventually enter and remain in $C$ for all future time. While  trivially met when $\S$ is compact, this notion of boundedness is, in general, too strong for stochastic models with non-compact $\S$. For example, theoretical population biologists often use models of the form $X_{t+1} = \xi_{t+1} X_t f(X_t)$ where $f$ is a positive function representing the survivorship and $\xi_t$ is a log-normally distributed random variable representing the mean number of offspring produced by an individual at time $t$. Since log-normal random variables are absolutely continuous on $(0,\infty)$, $X_t$ can become arbitrarily large with positive probability at any time step. While  one might argue that  this behavior is biologically unrealistic, models of this variety have provided many important biological insights and therefore deserve a careful mathematical treatment. 

One less restrictive notion of boundedness is that the probability  $X_t$ gets arbitrarily large becomes vanishingly small~\citep{meyn-tweedie-09}: 
\begin{definition} The Markov chain \eqref{eq:background1} is \emph{bounded in probability} if for all $\epsilon>0$ there exists a compact set $C\subset \S$ such that 
\begin{equation}
\liminf_{t\to\infty} \P[ X_t \in C|X_0=x] \ge 1-\epsilon
\end{equation}
for all $x\in \S$. 
\end{definition}
This definition of boundedness implies that across many realizations of the dynamics of \eqref{eq:background1}, there is a small probability that populations lie outside a compact set far in the long term. The standard definition of bounded in probability allows for the compact sets $C$ to depend on $x$ as well as $\epsilon$.  However, for most applications, it is more natural to require this stronger definition which ensures tendency for remaining bounded is independent of initial conditions. 

Alternatively, the empirical measure point of view  insists that the fraction of time the populations spend at arbitrarily high densities becomes vanishingly small.
\begin{definition} The Markov chain \eqref{eq:background1} is \emph{almost surely bounded on average} if for all $\epsilon>0$ there exists a compact set $C\subset \S$ such that 
\begin{equation}
\liminf_{t\to\infty} \Pi_t[C] \ge 1-\epsilon \mbox{ almost surely}
\end{equation}
whenever $X_0=x\in \S$. 
\end{definition}
The \emph{average} in this definition refers to the temporal average in the definition of the empirical measures $\Pi_t$. This terminology follows from a weaker notion of boundedness used in the Markov chain literature~\citep{meyn-tweedie-09}:
\begin{definition} The Markov chain \eqref{eq:background1} is \emph{bounded in probability on average} if for $\epsilon>0$ there exists a compact set $C\subset \S$ such that 
\begin{equation}
\liminf_{t\to\infty} \frac{1}{t}\sum_{i=0}^{t-1} \P[X_t \in C|X_0=x] \ge 1- \epsilon
\end{equation} 
for all $x\in \S$. 
\end{definition}
\noindent Both boundedness in probability and almost surely bounded on average imply bounded in probability on average. However, boundedness in probability need not imply almost sure boundedness on average, and vice-versa. Since we can not expect, in general as discussed earlier, that $X_t$ asymptotically remains in a compact set with probability one, \emph{I will refer to almost surely bounded on average as simply almost surely bounded}. 

As in the case of deterministic models, a practical method for verifying both forms of boundedness is finding an appropriate Lyapunov-type function. Recall a that function $V:\S \to \R_+$ is called \emph{proper} if $\lim_{\|x\| \to\infty } V(x) = +\infty$. The following theorem shows that  boundedness follows if there is a proper function $V$ decreasing, on average, along population trajectories whenever population densities are high.  

\begin{theorem}\label{thm:bounded}
Let $V:\S\to\R_+$ be a continuous, proper function. If there exist Borel functions $\alpha,\beta: \Omega \to \R_+$  such that 
\begin{equation}\label{eq:bounded}
V(F(x,\omega))\le \alpha(\omega) V(x)+ \beta(\omega) \mbox{ for all }\omega, x, 
\end{equation}   
$\E[\log \alpha(\xi_t)]<0$, $\E[\log^+\alpha(\xi_t)]<\infty$, and $\E[\log^+ \beta(\xi_t)]<\infty$ where $\log^+(z)=\max\{\log(z),0\}$, then \eqref{eq:background1} is bounded in probability and almost surely bounded. 
\end{theorem}

\begin{remark} An alternative proof, to the one given below, for the case of almost sure boundedness  was given by \citet[Proposition 4]{tpb-09}.
\end{remark}

\begin{proof} Define 
\[
Y_t = V(X_t), \mbox{  } \alpha_t=\alpha(\xi_t) \mbox{, and }\beta_t =\beta(\xi_t).
\]
Define $Z_t$ iteratively by $Z_0=Y_0$ and 
\[
Z_{t+1}=\alpha_{t+1} Z_t + \beta_{t+1}
\]
Theorem 2.1 in \citep{diaconis-freedman-99} implies there exists a non-negative random variable $\widehat Z$ such that $Z_t$ converges in probability to $\widehat  Z$ and the empirical measures $\frac{1}{t}\sum_{s=1}^t \delta_{Z_s}$ converge almost surely to the distribution of $\widehat Z$. 

Equation \eqref{eq:bounded} and the definition of $Z_t$ implies that $Z_t \ge Y_t \ge 0$ for all $t$. Given $\epsilon>0$, choose $a>0$ such that $\P[\widehat Z \in [0,a]]\ge 1-\epsilon/2$. Since $Z_t$ converges in probability to $\widehat Z$, there exists $T>0$ such that
\[
\P[X_t\in V^{-1}([0,a])] = \P[Y_t \in [0,a]] \ge \P[Z_t \in [0,a]] \ge 1-\epsilon 
\]
for all $t\ge T$. Since $V$ is proper, $V^{-1}([0,a])$ is compact. Therefore, $X_t$ is bounded in probability. 

Similarly, with probability one, 
\begin{eqnarray*}
\liminf_{t\to\infty} \Pi_t ( V^{-1}([0,a])) &=& \liminf_{t\to\infty} \frac{1}{t} \sum_{s=1}^t \delta_{Y_s}([0,a])\\
&\ge & \liminf_{t\to\infty} \frac{1}{t} \sum_{s=1}^t \delta_{Z_s}([0,a])=\P[\widehat Z\in [0,a]] \ge 1-\epsilon/2
\end{eqnarray*}
Therefore $X_t$ is almost surely bounded. 
\end{proof}

Boundedness in probability on average combined with the Feller property ensures the existence of an \emph{invariant probability measure}: a Borel probability measure $\mu$ on $\S$ such that if $X_0$ is distributed according to $\mu$ (i.e. $\P[X_0\in A] = \mu(A)$ for all Borel sets $A\subset \S$), then it is distributed according to $\mu$ for all time i.e. $\P[X_n\in A]=\mu (A)$ for all Borel $A\subset \S$. The proof of the following proposition follows from a standard argument, see e.g. \citet[Proposition 6.1.8]{duflo-97}. One can think of it as the stochastic analog of the fact that the $\omega$-limit set for a point is non-empty for dissipative maps. 

\begin{proposition}\label{prop:inv} If the Markov chain \eqref{eq:background1} is  bounded in probability on average, then the set of weak* limit points of $\frac{1}{t} \sum_{s=0}^{t-1} \P[X_s \in \cdot| X_0=x]$ with $x\in \S$ is non-empty and each of these limit points is an invariant probability measure. Alternatively, if the Markov chain \eqref{eq:background1} is almost surely bounded, then the set of weak* limit points of $\Pi_t$ with $X_0=x\in \S$ is almost surely non-empty and each of these limit points is  an invariant probability measure. \end{proposition}

\subsection{Persistence}

When the population dynamics are bounded, population biologists are often interested understanding the conditions ensuring the long-term persistence of the populations.  A natural analog of uniform persistence for stochastic models is given below. For these definitions, it useful to introduce the set of the population states within $\eta>0$ of extinction 
\[
\S_\eta = \{x\in \S: d(x,\S_0) \le \eta\}.
\] 
where $d(x,\S_0)=\min_{y\in \S_0} \| x-y\|$. 

From the ``ensemble'' point of view, the following notion of persistence was introduced by \citet{chesson-82}.
\begin{definition} The Markov chain \eqref{eq:background1} is \emph{persistent in probability} if for all $\epsilon>0$ there exists $\eta>0$ such that 
\begin{equation}
\limsup_{t\to\infty} \P[ X_t \in  \S_\eta |X_0=x] \le\epsilon
\end{equation}
for all $x\in \S\setminus \S_0$. 
\end{definition}
\noindent This definition asserts that reaching low densities or frequencies is very unlikely  in the long term.  The next definition provides the ``typical trajectory'' perspective on persistence. 
\begin{definition} The Markov chain \eqref{eq:background1} is \emph{almost surely persistent} if for all $\epsilon>0$ there exists $\eta>0$ such that 
\begin{equation}
\limsup_{t\to\infty} \Pi_t[ \S_\eta ] \le \epsilon \mbox{ almost surely}
\end{equation}
whenever $X_0=x\in \S\setminus \S_0$. 
\end{definition}
\noindent This definition asserts that the fraction of time a typical population trajectory spends near extinction states is very small.

When \eqref{eq:background1} is  bounded, Proposition~\ref{prop:inv} implies that there exists an invariant probability measure. If in addition \eqref{eq:background1} is persistent, then the following proposition implies that there exists a \emph{positive invariant probability measure} $\mu$ i.e. an invariant probability measure satisfying $\mu(\S_0)=0$. 

\begin{proposition}\label{prop:inv+} If the Markov chain \eqref{eq:background1} is  persistent in probability and bounded in probability, then the set of weak* limit points of $\frac{1}{t} \sum_{s=0}^{t-1} \P[X_s \in \cdot| X_0=x]$ with $x\in \S\setminus \S_0$ is non-empty and each of these limit points is a positive, invariant measure. Alternatively, if the Markov chain \eqref{eq:background1} is almost surely  persistent and almost surely bounded , then the set of weak* limit points of $\Pi_t$ with $X_0=x\in \S\setminus \S_0$ is almost-surely non-empty and each of these limit points is  almost-surely a positive, invariant measure. \end{proposition}

\begin{proof}
Suppose that the Markov chain \eqref{eq:background1} is persistent in probability and bounded in probability. Let $x\in \S \setminus \S_0$ and assume that $t_k\uparrow \infty$
is such that $\frac{1}{t_k} \sum_{s=0}^{t_k-1} \P[X_s \in \cdot| X_0=x]$ converges in the weak* topology to $\mu$. Proposition~\ref{prop:inv} implies that $\mu$ is invariant. On the other hand, given any natural number $n$,  persistence in probability implies there exists $\eta_n>0$ and $T>0$ such that 
\[
\P[ X_t \in  \S_{\eta_n} |X_0=x]<1/n
\]
for $t\ge T$. Hence, by weak* convergence $\mu(\S_{\eta_n})\le 1/n$ and $\mu(S_0)\le \limsup_{n\to\infty}\mu(S_{\eta_n}) =0$. The proof for the case of almost sure persistence  and almost sure boundedness  is similar.  
\end{proof}

When  a unique positive invariant probability measure exists and the system is  persistent, one can often show that  if $X_0=x\in \S \setminus \S_0$, then the distribution of $X_t$ converges to $\mu$  and $\Pi_t$ converges almost surely to $\mu$.   A powerful tool for verifying this stronger form of persistence is the following theorem due to ~\citet[Chapter 15]{meyn-tweedie-09}. This theorem relies on the concept of \emph{$\varphi$-irreducibility with respect to a Borel set $B\subset \S$}: there exists a Borel measure $\varphi$ on $B$ such that $\varphi(A)>0$ implies that $\P[X_n \in A$ for some $n| X_0=x]>0$ for all $x\in B$.

\begin{theorem} \label{thm:ergodic} Assume the Markov chain \eqref{eq:background1} is $\varphi$-irreducible on $\S\setminus\S_0$ and there exists a positive function $V: \S\setminus \S_0 \to \R_+$, a compact set $C\subset \S\setminus \S_0$, and constant $\beta>0$ such that 
\[
\E[ V(X_1)| X_0=x] \le (1-\beta) V(x) + \mathbf{1}_C(x) \mbox{ for all }x\in \S\setminus \S_0
\]
where $\mathbf{1}_C$ is the indicator function for $C$ i.e. $\mathbf{1}_C(x)=1$ if $x\in C$ and $0$ otherwise.  
Then there exists a unique positive invariant probability measure $\mu$ and the distribution of $X_t$ converges in the weak* topology to $\mu$. Moreover, $\Pi_t$ almost surely converges in the weak* topology to $\mu$ whenever $X_0 =x \in \S \setminus \S_0$. 
\end{theorem}

A drawback of requiring $\varphi$-reducibility is that it can be  difficult to verify or demonstrably false for important classes of tractable models. For instance, many important biological insights have been gleamed from models where there are a finite number of environmental states (i.e. $\Omega$ is a finite set). These models rarely satisfy the irreducibility condition and, consequently, may not have a unique positive invariant measure. 

In the next three sections, I review results for persistence of scalar single species models, multiple species models, and structured species models (e.g. spatial, age, or size structure).

\section{Scalar models} 

The simplest forms of the Markov chain \eqref{eq:background1} are the scalar models describing the dynamics of an unstructured, closed population i.e. $k=1$, $\S=[0,\infty)$, $\S_0=\{0\}$, in which case \eqref{eq:background1} simplifies to 
\begin{equation}\label{eq:scalar}
X_{t+1}=f(X_t, \xi_{t+1}) X_t \mbox{ with } X_t\in [0,\infty)
\end{equation}
where $f(x,\omega):[0,\infty)\times \Omega \to(0,\infty)$ is a continuous function. Random difference equations of this form have been studied extensively by many authors~\citep{athreya-dai-00,athreya-dai-02,athreya-schuh-03,bezandry-etal-08,bhattacharya-majumdar-04,ble-etal-07,chesson-82,ellner-84,fagerholm-hognas-02,gyllenberg-etal-94a,haskell-sacker-05,vellekoop-hognas-97}. Here, I focus on results that relate to persistence and boundedness. 

Reasonably general criterion for  extinction, persistence, and population explosion are given by the following Theorem. The proof of extinction and explosion follows from standard arguments that have been used by many authors~\citep{chesson-82,ellner-84,fagerholm-hognas-02,gyllenberg-etal-94b,vellekoop-hognas-97}. The argument for  boundedness  follows from Theorem~\ref{thm:bounded} and persistence follows from the univariate version of Theorem 1 in \citep{jmb-11}.

\begin{theorem}\label{thm:scalar} Assume $f(x,\omega)$ is a positive decreasing function in $x$ and  $\E[\log^+ f(0,\xi_t)]<\infty$. Then
\begin{enumerate}
\item[(i)] if $\E[\log f(0,\xi_t)]<0$,  then $\lim_{t\to\infty} X_t=0$ for all $x\ge 0$,
\item[(ii)] if $\lim_{x\to\infty} \E[\log f(x,\xi_t)]>0$,  then $\lim_{t\to\infty} X_t=\infty$ for all $x>0$, and
\item[(iii)]  if $\E[\log f(0,\xi)]>0$ and $\lim_{x\to\infty} \E[\log f(x,\xi_t)]<0$, then \eqref{eq:scalar} is bounded in probability, almost surely bounded, and almost surely persistent . 
\end{enumerate}
\end{theorem} 

The assumption that $f(x,\omega)$ is decreasing with $x$ holds for many ``classical'' single species models. However, this assumption doesn't hold for species exhibiting an Allee effect~\citep{courchamp-etal-99}. \citet{ellner-84}  proved results for models where $x\mapsto f(x,\omega)$ is not monotonic and $x\mapsto xf(x,\omega)$ is  increasing.

\begin{proof}
Assume that $\E[\log f(0,\xi)]<0$ and $X_0>0$. Then by the Strong Law of Large Numbers, 
\begin{eqnarray*}
\limsup_{t\to\infty} \frac{1}{t}\log X_t &=& \limsup_{t\to\infty}\frac{1}{t}\left(\sum_{s=0}^{t-1} \log f(X_s,\xi_{s+1})+\log X_0\right)\\
&\le & \lim_{t\to\infty}\frac{1}{t}\left(\sum_{s=0}^{t-1} \log f(0,\xi_{s+1})+\log X_0\right)=\E[\log f(0,\xi_t)]<0\\
\end{eqnarray*}
with probability one. Hence, $\lim_{t\to\infty} X_t =0$ with probability one. 

Assume that $\lim_{x\to\infty}\E[\log f(x,\xi)]>0$ and $X_0>0$. Then by the Strong Law of Large Numbers, 
\begin{eqnarray*}
\liminf_{t\to\infty} \frac{1}{t}\log X_t &=& \liminf_{t\to\infty}\frac{1}{t}\left(\sum_{s=0}^{t-1} \log f(X_s,\xi_{s+1})+\log X_0\right)\\
&\ge & \lim_{t\to\infty}\frac{1}{t}\left(\sum_{s=0}^{t-1} \lim_{x\to\infty}\log f(x,\xi_{s+1})+\log X_0\right)=\lim_{x\to\infty}\E[\log f(x,\xi_t)]>0\\
\end{eqnarray*}
with probability one. Hence, $\lim_{t\to\infty} X_t = \infty$ with probability one. 

Assume that $\E[\log f(0,\xi)]>0$ and $\lim_{x\to\infty} \E[\log f(x,\xi_t)]<0$. To verify boundedness, choose $M>0$ and $\epsilon>0$ such that $\E[\log f(x,\xi_t)]\le -\epsilon$ for all $x\ge M$. Define $V:[0,\infty)\to[0,\infty)$ to be the identity function $V(x)=x$. Define $\alpha_{t}=f(M,\xi_{t+1})$. Define $\beta_{t}=f(0,\xi_{t})M$. Since $f$ is decreasing in $M$, $X_{t+1}\le \alpha_{t+1}X_t$ whenever $X_t\ge M$. On the other hand,  $X_{t+1}\le f(0,\xi_{t+1})M =\beta_{t+1}$ whenever $X_t\le M$. Since $\E[\log^+ \beta_t ] <\infty$ by assumption and $\E[\log \alpha_t] \le -\epsilon$ by definition, Theorem~\ref{thm:bounded} implies that $X_t$ is bounded in probability and almost surely bounded . The proof of  almost surely persistent  follows verbatim as in the proof of  \citep[Theorem 1]{jmb-11}. The compactness assumption in this proof is only needed for $\S_0$ which equals $\{0\}$ for \eqref{eq:scalar}.
\end{proof}

Theorem~\ref{thm:scalar} suggests the following conjecture. 

\begin{conjecture}
Under the same assumptions of Theorem~\ref{thm:scalar}(iii), \eqref{eq:scalar} is persistent in probability. 
\end{conjecture}

\begin{example}[The stochastic Hasell model]\label{ex:hasell}{\rm 
To illustrate the utility of Theorem~\ref{thm:scalar},  consider \citet{hassell-75}'s single species model in which $\omega=(\lambda,b)$ and 
\[
f(x,\omega)=\frac{\lambda}{(1+x)^b}
\]
where $\lambda$ is the intrinsic fitness of an individual, and $b>0$ determines the strength of intraspecific competition. For this model, \[
\E[\log f(x,\xi_t)]=\E[\log\lambda_t]-\E[b_t]\log(1+x)\] 
Hence, $\lim_{x\to\infty}\E[\log f(x,\xi_t)]=-\infty$ and $\E[\log f(0,\xi_t)]=\E[\log \lambda_t]$. Theorem~\ref{thm:scalar} implies almost sure extinction if $\E[\log \lambda_t]<0$ and almost sure persistence   if $\E[\log \lambda_t]>0$.$\blacksquare$ }
\end{example}

To arrive at stronger conclusions, more assumptions about either the form of the nonlinearity or the noise are necessary. With respect to the form of the nonlinearity, \citet[Theorem 2.2]{ellner-84} showed that a monotonicity assumption on $xf(x)$ ensures converge in distribution to a positive random variable. In the case of the stochastic Hassell model described in Example~\ref{ex:hasell}, this convergence occurs whenever $b_t \in (0,1]$ and $\E[\log \lambda_t]>0$. \\

\begin{theorem}[Ellner 1984]~\label{thm:ellner} Assume that $F(x,\omega)=xf(x,\omega)$ is continuously differentiable and strictly increasing in $x$, $f(x,\omega)$ is strictly decreasing in $x$, and $\E[\log f(x,\xi_t)]<\infty$ for some $x>0$. If $\E[\log f(0,\xi_t)]>0$ and $\lim_{x\to\infty} \E[\log f(x,\xi_t)]<0$, then there exists a positive invariant probability measure $\mu$ and the distribution of $X_t$ converges in the weak* topology to $\mu$ whenever $X_0=x>0$.
\end{theorem}

The assumption that $x\mapsto F(x,\omega)$  is strictly increasing implies $F(\cdot,\omega)$ is a monotone map for each $\omega$. For deterministic systems, this monotonicity provides a lot of leverage to understand the map dynamics even in higher dimensions, as reviewed in \cite{hirsch-smith-05}. This leverage has been extended to random maps as reviewed by \citet{chueshov-02}. In the special cases of the stochastic Beverton-Holt model $X_{t+1}=\frac{\lambda_{t+1}}{1+a_{t+1}X_t}X_t$  and the stochastic Beverton-Holt model with survivorship $X_{t+1}=\frac{\lambda_{t+1}}{1+a_{t+1}X_t}X_t+s_{t+1}X_t$, a similar result to Theorem~\ref{thm:ellner} was proven by \cite{haskell-sacker-05} and \cite{bezandry-etal-08}, respectively.

When $b_t>1$ for the stochastic Hassell model, monotonicity fails (i.e. the map $x\mapsto x\,f(x,\omega)$ is unimodal) and other assumptions are necessary to ensure convergence to a positive random variable. \citet{vellekoop-hognas-97} proved an ergodic form of  persistence by placing stronger assumptions on the random variables $\xi_t$. Their result is applicable to the Hassell model under the assumption that  $\lambda_t$ is constant. The proof  uses the Lyapunov function characterization of ergodicity described in Theorem~\ref{thm:ergodic}. 

\begin{theorem}[Vellekoop and H\"{o}gn\"{a}s 1997] \label{thm:vh} Assume that 
\[
f(x,\omega)=\lambda\,g(x)^{-\omega}
\]
where $g$ is a positive differentiable function satisfying $x\mapsto xg'(x)/g(x)$ is strictly increasing on $[0,\infty)$. Assume that $\xi_t$ are i.i.d. positive random variables with $\E[\xi_t],\E[\xi_t^2]<\infty$ and a positive  density on $(0,\infty)$. If $\lambda>1$, then there exists a positive invariant probability measure $\mu$, the distribution of $X_t$ converges to $\mu$ whenever $X_0=x>0$, and the empirical measures $\Pi_t$ converge almost surely to $\mu$ whenever $X_0=x>0$.
\end{theorem}

Since $g(x)=1+x$ satisfies $xg'(x)/g(x)=x/(1+x)$, Theorem~\ref{thm:vh} applies to the stochastic Hassell model with $\lambda_t$ constant and $\xi_t=b_t$.  This theorem is also applicable to the stochastic Ricker equation $X_{t+1}=X_t \exp(r-a_{t+1}X_t)$ where $r>0$ is the intrinsic per-capita growth rate of the population and  $\xi_t = a_t$ measures the intensity of interspecific competition. \citet{gyllenberg-etal-94b} studied the stochastic Ricker model when either $r_t$ or $a_t$ vary randomly. More recently, \citet{fagerholm-hognas-02} studied the dynamics of the stochastic Ricker model when both $r_t$ and $a_t$ vary randomly. Quite surprising, they prove that if $E[r_t]=0$, then \eqref{eq:scalar} is \emph{null recurrent}: there exists no positive  invariant probability measure despite $\P[X_t \in A$ infinitely often $]=1$ for all Borel sets $A$ with positive Lebesgue measure. 

\begin{theorem}[Fagerholm \& H\"{o}gnas 2002]\label{thm:fh} Consider the stochastic Ricker model $X_{t+1}=X_t \exp(r_{t+1}-a_{t+1}X_t)$ where 
\begin{itemize}
\item $r_1,r_2,\dots$ is a sequence of i.i.d. random variables such that $\E[r_t]<\infty$ and $r_t$ has a positive density on $(-\infty,\infty)$,
\item $a_1,a_2,\dots$ is a sequence of positive i.i.d. random variables independent of $r_t$ such that $\E[a_t]<\infty$, and
\item there exists $x_c>0$ such that $\E[\exp(r_1 x)]<\infty$ for all $0\le x \le x_c$. 
\end{itemize}
Then one of the following statements holds
\begin{description}
\item[extinction] If $\E[r_t]<0$, then $X_t$ converges to  $0$ with probability one, 
\item [null recurrence] if $\E[r_t]=0$ and $X_0>0$, then $X_t$ is null recurrent, or
\item [persistence] if $\E[r_t]>0$, then there exists a positive invariant measure $\mu$ such that the distribution of $X_t$ converges to $\mu$ in the weak* topology. 
\end{description} 
\end{theorem} 

\citet{fagerholm-hognas-02} and \citet{gyllenberg-etal-94b} also studied the case when $a_t\le 0$ with positive probability. In this case the dynamics have the potential to be explosive (i.e. the intraspecific interactions enhance growth) and, consequently, the analysis is more subtle. This case was motivated by numerical studies where $r_t$ and $1/a_t$ were normally distributed.

\section{Multiple species interactions}

For multiple species interactions in fluctuating environments, one can consider models of the form 
\begin{equation}\label{eq:multi}
X_{t+1}^i = f_i(X_t,\xi_t)X^i_t \mbox{ with }i=1,\dots,k
\end{equation}
where $X_t^i$ is the density or frequency of species $i$ at time $t$. For these models $\S$ is a closed subset of $\R^k_+=\{x\in \R^k: x_i \ge 0\}$ and $\S_0 = \{x\in \S: \prod_i x_i =0\}$ corresponds to the extinction of one or more  species.  These models have been used extensively to understand under what conditions environmental stochasticity and species interactions facilitate or disrupt species or genetic diversity~\citep{chesson-78,chesson-warner-81,chesson-82,chesson-94,kuang-chesson-09,anderies-beisner-00,turelli-78,turelli-81}. 

Despite extensive numerical and theoretical work, there are (to the best of my knowledge)  only two sets of mathematical results concerning persistence for these multispecies models. The first set of results~\citep{chesson-ellner-89,ellner-89} applies to two species competitive systems. The second set of results~\citep{jmb-11} apply to $k$ species systems provided the dynamics satisfy an appropriate compactness assumption. Both of these results  utilize the average per-capita growth rates of populations when rare. More specifically,  define the \emph{mean per-capita growth rate of species $i$ at population state $x$} to be 
\begin{equation}\label{eq:growth1}
\lambda_i(x)=\E[\log f_i(x, \xi_t)],
\end{equation}
and define the \emph{mean per-capita growth rate of species $i$ at invariant probability measure $\mu$} to be 
\begin{equation}\label{eq:growth2}
\lambda_i(\mu)=\int \lambda_i(x) \mu(dx).
\end{equation}
When $\mu$ is ergodic (i.e. $\int \E[h(F(x,\xi_1)]\,\mu(dx)=\int h(x)\mu(dx)$ if and only if $h$ is a constant function $\mu$-almost surely), $\lambda_i(\mu)$ is the long-term average of the per-capita growth rate of population $i$ when in a system supported by the invariant measure $\mu$. More precisely, since each of the sets $\{x_i = 0\}$ is invariant under the  dynamics in (\ref{eq:multi}), there exists a set $\supp(\mu) \subset \{1,\ldots, k\}$ such that for $\mu$-almost all $x$, $x_i > 0$ if and only if $i \in \supp(\mu).$ One can interpret $\supp(\mu)$ as the set of populations supported by $\mu$.

\subsection{Competition between two species}

For competitive interactions between two species, \citet{chesson-ellner-89} and \citet{ellner-89} proved a ``mutual invasibilty'' condition implies stochastic persistence. In all of their results, they built on the single species results by assuming 
\begin{description}
\item[B1] for each $i=1,2$, there exists a positive invariant measure $\mu_i$ such that the distribution of $X^i_t$ converges to $\mu_i$ in the weak* topology whenever $X_0^i>0$ and $X_0^j=0$ for $j\neq i$, and
\item[B2] the mean per-capita growth rates $\lambda_i(x)$ are continuous functions of $x\in\S$.~\footnote{This assumption is made in \citet{ellner-89} and is met for most models. A slighter weaker assumption is made in \citet{chesson-ellner-89}.} 
\end{description}
The mutual invasibilty condition~\citep{turelli-81} asserts that the species coexist provided that $\lambda_1(\mu_2)>0$ and $\lambda_2(\mu_1)>0$. Intuitively, whenever one species, say species $2$, is rare, the dynamics of the other species approaches its invariant measure $\mu_1$. At this invariant measure, the per-capita growth rate of species $2$ is positive (i.e. $\lambda_2(\mu_1)>0$) and, consequently, increases in abundance. Since each species increases in abundance when rare, they coexist. Under the assumption that the competitive dynamics are monotonic, \citet{chesson-ellner-89} proved that mutual invasibilty implies stochastic persistence in probability.  \\

\begin{theorem}[Chesson and Ellner 1989]~\label{thm:ce} Assume the Markov chain \eqref{eq:multi} with $k=2$ satisfy \textbf{B1-B2},  and
\begin{itemize}
\item the  equations  $x_1=F_1((x_1,0),\xi_t)$ and $x_2=F_2((0,x_2),\xi_t)$ hold with probability one only for $x_1=0$ and $x_2=0$, respectively, and
\item the functions $F_i(x,\omega)$ are non-decreasing in $x_1$ and $x_2$, and positive whenever $x_i>0$.  
\end{itemize}
Then $\lambda_1(\mu_2)>0$  and $\lambda_2(\mu_1)>0$ implies  \eqref{eq:multi} is persistent in probability.
\end{theorem}

It is natural to make the following conjecture. This conjecture follows from Theorem~\ref{thm:persist} whenever the dynamics of \eqref{eq:multi} asymptotically enter a compact set. 

\begin{conjecture} Under the conditions of Theorem~\ref{thm:ce}, \eqref{eq:multi} is almost surely persistent . \end{conjecture}

Under a stronger assumption about the noise terms $\xi_t$ in Theorem~\ref{thm:ce}, \citet{ellner-89} proved that there exists a unique positive invariant measure $\mu$ such that the distribution of $X_t$ converges to $\mu$ whenever $X_0^1>0$ and $X_0^2>0$. 

\begin{example}{\rm
\citet{chesson-ellner-89} illustrated the applicability of Theorem~\ref{thm:ce} with the following competition model
\begin{eqnarray*}
X^1_{t+1}&=& \frac{ \xi^1_{t+1} X_t^1}{1+\xi^1_{t+1}X^1_t+\xi_{t+1}^2 X^2_t}+a X_t^1\\
X^2_{t+1}&=& \frac{ \xi^2_{t+1} X_t^2}{1+\xi^1_{t+1}X^1_t+\xi_{t+1}^2 X^2_t}+a X_t^2
\end{eqnarray*}
where $\xi_t^i>0$ represents the per-capita fecundity of species $i$ and $0<a<1$ represents the fraction of individuals surviving to the next time step. $\xi^1_t,\xi^2_t$ are assumed to have an exchangeable joint distribution (i.e. $\P[(\xi_t^1,\xi_t^2) \in A]=\P[(\xi_t^2,\xi_t^1) \in A]$ for any Borel set $A\subset \R^2_+$.). If $\E[\log (\xi_t^1+a)]>0$, then Theorem~\ref{thm:ellner} implies that  \textbf{B1} is satisfied. By the exchangeable assumption, $\lambda_1(\mu_2)=\lambda_2(\mu_1)$. 
Following \citet{chesson-88}, define 
\[
g(\omega)=\log\left(\frac{\omega_1}{1+\omega_2x_2}+a\right).
\]
\citet{chesson-88} proved that  
\[
\lambda_1(\mu_2)= -\frac{1}{2} \E\left[\int_{\xi^1_t}^{\xi^2_t} \int_{\xi^1_t}^{\xi^2_t}\frac{\partial^2 g(\omega)}{\partial \omega_1\omega_2} d\omega_1 d\omega_2 \right].
\]
Since 
$
\frac{\partial^2 g(\omega)}{\partial \omega_1\omega_2}
<0$, it follows that $\lambda_1(\mu_2)=\lambda_2(\mu_1)>0$. $\blacksquare$
}
\end{example}

The monotonicity assumption of $F$ is too strong to cover all models of competitive interactions. For instance, Theorem~\ref{thm:ce} does not apply to stochastic Ricker models of competition (see Example~\ref{ex:ricker-comp} below). However, \citet{ellner-89} proved that under a stronger assumption on the noise terms $\xi_t$, it is possible to show stochastic persistence for such models. The theory presented in the next section provides a similar approach to verifying persistence for these models. 

\begin{theorem}[Ellner 1989]\label{thm:ellner89} Assume \textbf{B1-B2} and 
\begin{itemize}
\item \eqref{eq:multi} is $\varphi$ irreducible on $(0,\infty)\times(0,\infty)$,
\item \eqref{eq:multi} is strongly continuous, i.e. for any measurable $A\subset \R^2_+$, 
$\P[X_1\in A|X_0=x_n]$ converges to $\P[X_1\in A|X_0=x]$ whenever $x_n \to x$, and 
\item for any $x\in \R^2_+$, $\sup_{t>0} \E[\log^+ X_t^i|X_0=x] <\infty$ for $i=1,2$. 
\end{itemize}
If $\lambda_1(\mu_2)>0$ and $\lambda_2(\mu_1)>0$, then there exists a unique positive invariant measure $\mu$ and the distribution of $X_t$ converges to $\mu$ in the weak* topology whenever $X_0^1>0$ and $X_0^2>0$. 
\end{theorem}

\begin{example}[Ricker equations of competition]\label{ex:ricker-comp}{\rm To illustrate the applicability of Theorem~\ref{thm:ellner89} consider a stochastic version of the Ricker equations of competition:
\begin{equation}\label{eq:ricker-comp}
X_{t+1}^i = X_t^i  \exp(\xi^i_{t+1} -X_t^i - \alpha_j X^j_t)\quad i,j=1,2; i\neq j
\end{equation}
where $\alpha_j>0$ are inter-specific competition coefficients and $\xi^i_{t+1}$ are normally  distributed intrinsic rates of growth with means $r_i$. Since $\exp(\xi_t^i)$ are log-normally distributed, they have a positive density on $(0,\infty)$. Theorem~\ref{thm:fh}  implies that \textbf{B1} holds whenever $r_i>0$ for $i=1,2$. In particular, both species are persistent in isolation under this assumption.  

The positive density of $(\xi_t^1,\xi_t^2)$ on $(0,\infty)\times (0,\infty)$ implies that  \eqref{eq:ricker-comp} is $\varphi$-irreducible on $(0,\infty)\times (0,\infty)$ with respect to Lebesgue measure. The strong-continuity condition  is easily verified. The boundedness condition follows from  Theorem~\ref{thm:bounded} with  $V(x)=x_1+x_2$. Let $\mu_i$ with $i=1,2$ be the invariant probability measures in assumption \textbf{B1}. Theorem~\ref{thm:ellner} implies  persistence in probability for the competing species whenever 
\[
\lambda_i(\mu_j)=r_i -\alpha_j \E[\widehat X^j]>0 \mbox{ for }i=1,2\mbox{ and }j\neq i
\]
where $\widehat X^j$ is random variable with law $\mu_j$. By invariance of $\mu_j$, 
\[
\E[\log \widehat X^j] = 
\E[\log X_1^j | X_0^j =\widehat X^j] =\E[\log(\widehat X_0^j \exp(\xi_1^j -\widehat X^j)) ]
=\E[\log \widehat X^j] + r_j -\E[\widehat X^j].
\]
Hence, $\E[\widehat X^j ] =r_j$ and coexistence occurs if
\[
r_i>\alpha_j r_j \mbox{ for }i=1,2\mbox{ and }j\neq i.
\]
Thus, the conditions for coexistence are the same for this stochastic version of the Ricker equations of competition and their deterministic counterpart.$\blacksquare$ }
\end{example}

\subsection{General multispecies}

Recently, \citet{jmb-11} extended the results of Chesson and Ellner to an arbitrary number of species by developing stochastic analogs of the classical permanence criteria for deterministic systems~\citep{hofbauer-81,hofbauer-sigmund-98}. These results are based on the following assumptions about \eqref{eq:multi}.
\begin{description}
\item [C1] There exists a compact set $\S$ of $\R^k_+=\{x\in\R^k: x_i\ge 0\}$ such that $X_t\in \S$ for all $t\ge 0$.
\item [C2] $f_i(x,\omega)$ are strictly positive functions,  continuous in $x$ and measurable in $\omega.$
\item [C3]  For all $i$, $\sup_{x\in \S}\E[(\log f_i(x,\xi_t))^2 ]<\infty$.
\end{description}
Assumption \textbf{C1} requires that the populations remain bounded for all time. Assumption \textbf{C2}  implies that $\{X_t\}_{t=0}^\infty$ is Feller. Assumption \textbf{C3} is a technical assumption met by many models.

Under these assumptions, \citet{jmb-11} proved that if every invariant measure supported by $\S_0$ can be invaded by some species, then the system is persistent. 

\begin{theorem}[Schreiber, Bena\"{\i}m and Atchade 2011]
\label{thm:persistence} Assume \textbf{C1-C3} and one of the following equivalent conditions hold:
\begin{enumerate}
\item[(i)] For all invariant probability measures $\mu$  supported on $ \S_0$,
\[
\lambda_*(\mu) := \max_i  \lambda_i(\mu)>0\mbox{, or}
\]
\item[(ii)]
There exists a positive vector $p=(p_1,\dots,p_k)>0$ such that
\begin{equation}\label{eq:condition}
\sum_i p_i \lambda_i (\mu) >0
\end{equation}
for all ergodic probability measures $\mu$  supported by $\S_0.$
\end{enumerate}
Then the Markov chain \eqref{eq:multi} is almost surely persistent.
\end{theorem}

While this theorem, as shown below, applies to many models, it has several limitations. First, it doesn't provide a statement about persistence in probability. None the less, it is natural to make the following conjecture.
\begin{conjecture}
Under the assumptions of Theorem~\ref{thm:persistence}, \eqref{eq:multi} is persistent in probability. 
\end{conjecture}
A second limitation of Theorem~\ref{thm:persistence} is that it requires dynamics asymptotically confined to a compact set. While this limitation as discussed earlier might be biologically realistic, it would be useful to have a result that applies to stochastically bounded systems. In particular, one could ask whether the following conjecture (or an appropriate modification of it) is true. 
\begin{conjecture}
Assume \eqref{eq:multi} is bounded in probability (respectively, almost surely), \textbf{C1-C3} hold, and $\lambda^*(\mu)>0$ for all invariant measures $\mu$ supported by $S_0$. Then \eqref{eq:multi} is  persistent in probability (respectively, almost surely). \end{conjecture}

Theorem~\ref{thm:persistence} does not ensure that there is a unique positive stationary distribution. For this stronger conclusion, there has to be sufficient noise in the system to ensure after enough time any positive population state can move close to any other positive population state. 

\begin{theorem}[Schreiber, Bena\"{\i}m and Atchade 2011]\label{thm:uniqueness} Assume that   $\{X_t\}$ is $\varphi$-irreducible over $\S \setminus \S_{\eta}$  for all $\eta > 0$, and that the assumption of Theorem~\ref{thm:persistence} holds. Then there exists a unique invariant probability measure $\mu$ such that $\mu (\S_0)=0$ and the occupation measures $\Pi_t$ converge almost surely to $\mu$ as $t\to\infty$, whenever $X_0=x\in \iS$.
\end{theorem}

Note that $\varphi$-irreducibility condition does not require that the same $\varphi$ is used for all $\eta>0$. Under a stronger irreducibility condition, \citet{jmb-11} proves that the distribution of $X_t$ converges to $\mu$ whenever $X_0\in \iS$.

\begin{example}[Coexistence of many competitors]\label{ex:lottery}{\rm To illustrate the applicability of Theorem~\ref{thm:uniqueness} to higher-dimensional models of competition, consider the lottery model of \citet{chesson-warner-81}. This model describes species requiring a territory or ``home'' (an area held to the exclusion of others) in order to reproduce. Moreover, space is always in short supply and, consequently, all patches are occupied. Let $X_t^i$ denote the fraction of space occupied by species $i$ at time $t$, $\xi_t^i$ the fecundity of species $i$ at time $t$, and $d$ the fraction of individuals dying each time step. Under these assumptions, the lottery model is given by 
\begin{equation}\label{lottery}
X_{t+1}^i = (1-d) X_t^i + d \frac{X_t^i \xi_{t+1}^i}{\sum_j X_t^j \xi_{t+1}^j} \quad i=1,\dots,k.
\end{equation}
Here $\S$ is the probability simplex $\{x\in \R^k_+: \sum_i x_i =1\}$. 
Let  $\log \xi_t^i$ be normally distributed with means $b_i>0$ and variances $\sigma_i^2\ge 0$. Furthermore, assume that $\xi_t^1,\dots,\xi^k$ are independent. Since log-normal distributions have a positive density on $(0,\infty)$, $\{X_t\}$ is $\varphi$-irreducible on $\S_\eta$ with respect to Lebesgue measure for all $\eta>0$. 

If  $b_1>\dots>b_k$ take on distinct values and there is no environmental noise (i.e. $\sigma_i=0$), then  species $1$ excludes all the remaining species, i.e. $\lim_{t\to\infty}X_t^i = 0$ for $i=2,\dots,k$ whenever $X_0^1>0$. To show how environmental stochasticity can alter this ecological outcome, consider the case that $b_1=\dots=b_k$ and $\sigma_i>0$ for all $i$.  Let $\mu$ be any invariant probability measure on $\S_0$. Choose a species $i$ such that $\mu(\{x\in \S: x_j>0 $ iff $j=i\})<1$. By Taylor's theorem,
\begin{equation}\label{taylor}
\lambda_i(\mu)=-d+d\int \E\left[\frac{\xi_t^i}{\sum_j x_j \xi_t^j}\right] \mu(dx) + \mathcal{O}(d^2).
\end{equation}
Independence and Jensen's inequality imply
\begin{equation}\label{jensen}
\int \E\left[\frac{ \xi_t^i}{\sum_j x_j \xi_t^j}\right] \mu(dx) > 
\E[\xi^i] \int \frac{ 1}{\sum_j x_j \E[\xi_t^j]} \mu(dx)=1
\end{equation}
where the final equality follows from the assumption that $b_1=\dots=b_k$. Combining equations \eqref{taylor} and \eqref{jensen} imply that 
\[
\lambda_i(\mu)>0
\]
provided that $d>0$ is sufficiently small. By compactness of the invariant probability measures on $\S_0$, it follows that $\lambda^*(\mu)>0$ for all invariant probability measures supported by $\S_0$ whenever $d>0$ is sufficiently small. Hence, Theorem~\ref{thm:uniqueness} implies there is a unique positive stationary measure and \eqref{lottery} is almost surely persistent. By continuity, these conclusions still apply whenever $\max_{i,j}|b_i-b_j|$ is sufficiently small. Therefore, environmental stochasticity can mediate coexistence between an arbitrary number of competitors. 
$\blacksquare$}
\end{example}

\begin{example}[Rock, paper, scissors]\label{ex:rps}{\rm In the basic lottery model described in Example~\ref{ex:lottery} per-capita reproductive rates are independent of species frequencies. Frequency-dependent feedbacks, however, can be quite important. To illustrate these feedbacks and to illustrate that persistence may require more than invasibilty by a missing species, consider a rock-paper-scissor version of the lottery model.  To model this intransitive interaction, we assume that the per-capita reproductive rates are linear functions of the species frequencies
\[
b_i(X_t,\xi_{t+1})= \sum_j \xi^{ij}_{t+1} X^j_t
\]
where
\[
\xi_t =  \begin{pmatrix} \beta_t & \alpha_t & \gamma_t  \cr \gamma_t & \beta_t & \alpha_t \cr \alpha_t & \gamma_t & \beta_t
\end{pmatrix}
\]
and $\alpha_t>\beta_t>\gamma_t>0$ for all $t$. The frequency-dependent lottery model becomes 
\begin{equation}\label{rps}
X_{t+1}^i = (1-d) X_t^i + d \frac{X_t^i b_i(\xi_{t+1},X_t)}{\sum_j X_t^j b_j(\xi_{t+1},X_t)} \quad i=1,2,3.
\end{equation}

For any pair of strategies, say $1$ and $2$, the dominant strategy, $1$ in this case,  displaces the subordinate strategy. Indeed, assume $X^3_0=0$. If $y_t=X^2_t/X^1_t$ and $z_t= \sum_i \xi^i_{t+1}X^i_t$, then
\[
y_{t+1}=\frac{(1-d)z_t+d(\gamma_{t+1} X^1_t+\beta_{t+1} X^2_t)}{(1-d)z_t+d(\beta_{t+1} X^1_t+\alpha_{t+1} X^2_t)}\,y_t<y_t
\]
is a decreasing sequence that converges to $0$. Hence, the only ergodic, invariant probability measures on $\Delta_0$ are Dirac measures $\delta_{x}$ supported on $x=(1,0,0),(0,1,0),(0,0,1)$.  At these ergodic measures, the invasion rates are given by
\begin{center}
\begin{tabular}{c|ccc}
$\mu$ & $\lambda_1(\mu)$ & $\lambda_2(\mu)$ &$\lambda_3(\mu)$ \\\hline
$\delta_{(1,0,0)}$& 0& $\E\left[\log \left( 1-d+d\,\alpha_t/\beta_t\right)\right]$& $\E\left[\log \left( 1-d+d\,\gamma_t/\beta_t\right)\right]$\\
$\delta_{(0,1,0)}$&  $\E\left[\log \left( 1-d+d\,\gamma_t/\beta_t\right)\right]$& 0& $\E\left[\log \left( 1-d+d\,\alpha_t/\beta_t\right)\right]$\\
$\delta_{(0,0,1)}$&$\E\left[\log \left( 1-d+d\,\alpha_t/\beta_t\right)\right]$ &  $\E\left[\log \left( 1-d+d\,\gamma_t/\beta_t\right)\right]$&0
\end{tabular}
\end{center}
A straightforward algebraic competition reveals that the conditions for persistence are satisfied if and only if
\begin{equation}\label{eq:rps}
\E\left[\log \left( 1-d+d\,\alpha_t/\beta_t\right)\right]+\E\left[\log \left( 1-d+d\,\gamma_t/\beta_t\right)\right]>0
\end{equation}
For small $d>0$, a Taylor's approximation similar to Example~\ref{ex:lottery} yields the following simpler condition for persistence:
\[ \E\left[\frac{\alpha_t}{\beta_t}\right]+\E\left[\frac{\gamma_t}{\beta_t}\right]>2.
\]
for $d>0$ sufficiently small. $\blacksquare$} \end{example}

I conjecture that if the opposite inequality of \eqref{eq:rps} holds, then persistence does not occur. More generally, 

\begin{conjecture} Assume that \textbf{C1-C3} hold, \eqref{eq:multi} is $\varphi$-irreducible over $\S\setminus \S_\eta$ for all $\eta>0$, and $\lambda^*(\mu)<0$ for all invariant measures $\mu$ supported on $S_0$. Then 
\[
\lim_{t\to\infty} \prod_i X_t^i =0
\]
with probability one.  \end{conjecture}

\citet{benaim-etal-08} proved the continuous-time version of this conjecture  for stochastic ODEs on the probability simplex with a small diffusion term. Without the $\varphi$-irreducibility assumption, this conjecture is definitively false.

\section{Structured populations}

Populations often consist of a heterogeneous mixture of  individuals in different states  such as its age, size, physiological condition, or location in space~\citep{caswell-01}. If the population consists of $k$ states, then its dynamics on $\S=\R^k_+$ can be described by nonlinear, stochastic matrix models of the form 
\begin{equation}\label{eq:structured}
X_{t+1}=A(X_t,\xi_{t+1})X_t
\end{equation}
where $X_t=(X_t^1,\dots,X_t^k)$ is the vector of population densities across the different states and $A(x,\omega)$ is a non-negative matrix whose entries represent transition probabilities, survivorship likelihoods, and fecundities. Here, the extinction set is $\S_0=\{0\}$.

When population abundances are low, it seems reasonable to approximate the dynamics of \eqref{eq:structured} with the linear equation
\[
X_{t+1}=A(0,\xi_{t+1})X_t
\]
in which case, 
\[
X_t= A(0,\xi_{t})\dots A(0,\xi_1)X_0.\]
Under suitable conditions (e.g.  $A(0,\xi_t)$ are primitive, and $\E( \left| \ln\|A(0,\xi_1)\| \right|)<\infty$), the work of \citet{ruelle-79} and \citet{arnold-etal-94} implies that there exists a quantity $\gamma$  such that 
\begin{equation}\label{lyap}
\lim_{t\to\infty}\frac{1}{t} \ln \left( X_t^1+\dots+X_t^k\right) = \gamma  \mbox{ with probability one.}
\end{equation}
In other words, the total population size $ X_t^1+\dots+X_t^k$ grows approximately like  $( X_0^1+\dots+X_0^k) e^{\gamma t}$. The quantity $\gamma$ is known as the \emph{dominant Lyapunov exponent} and is also known as the \emph{stochastic growth rate} in theoretical ecology~\citep{tuljapurkar-90,caswell-01}. For the linearized model, if $\gamma>0$, the population grows exponentially and  persists. Alternatively, if $\gamma<0$, the population is driven to extinction. 

To contend with the nonlinearities in structured population models, \citet{hardin-etal-88a} extended the work of \citet{ellner-84} to structured populations in serially uncorrelated environments. Under a slightly different set of assumptions, \citet{tpb-09} proved a similar result that also addresses convergence of the empirical measures and also applies to asymptotically stationary environments. Since these latter assumptions are slightly easier to present, I will focus on them. In addition to standing assumptions \textbf{A1-A3}, the following additional assumptions are needed.  

\begin{description}
\item [D1: Primitivity] There is a positive integer $T$ such that  with probability one
\[
A(X_t,\xi_t)\dots A(X_1,\xi_1) \mbox{ has all positive entries}
\]
whenever $X_1\in \R^k_+$ and $t\ge T$.
\item [D2: Smoothness] The map $(x,\e)\to A(x,\e)$ is Borel, $x\mapsto A(x,\omega)x=F(x,\omega)$ is twice continuous differentiable for all $\e\in \S$, $x\in\Rkp$, and 
\[
\E{\left(\sup_{\|x\|\le 1}\ln^+ \left( \|F(x,\e)\|+\|D_xF(x,\e)\|+\|D^2_{x} F(x,\e)\|\right)\right)}<+\infty
\]
\item [D3: Intraspecific competition] The matrix entries
$A_{ij}(x,\e)$ satisfy
\[
\frac{\partial A_{ij}}{
\partial x_l}(x,\e) \le 0
\]
for all $\e$ and $x$. Moreover, for  each $i$ there exists some $j$ and $l$ such that this
inequality is strict for all $\e$ and $x$.
\item [D4: Compensating density dependence]
All entries of the derivative of $x\mapsto A(x,\omega)x$ are
non-negative for all $\omega$ and $x$.
\end{description}

Assumption \textbf{D1}, roughly translated, asserts that after sufficiently many time steps, individuals in every state contribute to the abundance of individuals in all other states. For constant environments (i.e. $\S$ consists of a single environmental state), this assumption corresponds to a matrix being primitive~\citep{horn-johnson-90,caswell-01}.   Assumption \textbf{D2}  is purely technical, but is met for most models. It is worth noting that \citet{hardin-etal-88a} do not require smoothness. Assumption \textbf{D3}  accounts for competition between the stages. Assumption \textbf{D4} ensures that \eqref{eq:structured} is a random, monotone dynamical system~\citep{chueshov-02}. 

When $\gamma$ is negative for the linearization, extinction is expected as the following proposition demonstrates. 

\begin{proposition}\label{prop:extinct} L
Assume \textbf{D1} and \textbf{D3} hold. Then 
\[
\limsup_{n\to\infty} \frac{1}{n}\ln \| X_t \|\le \gamma 
\]
with probability one and where $\gamma$ is the Lyapunov exponent defined in \eqref{lyap}. In particular, if $\gamma<0$, then 
\[
\lim_{n\to\infty} X_t =0 
\]
with probability one. 
 \end{proposition}

 When $\gamma$ is positive for the linearization, stochastic persistence is expected as the following theorem shows. 
 
 \begin{theorem}\label{thm:persist}
Assume \textbf{D1}--\textbf{D4} and the assumptions of Theorem~\ref{thm:bounded} hold (i.e. the system is bounded). If $\gamma>0$, then there exists an invariant probability measure $\mu$ such that $\mu(\S_0)=0$, the distribution of $X_t$ converges to $\mu$ whenever $X_0=x\in \S\setminus\S_0$. Moreover, with probability one, $\Pi_t$ converges toward $\mu$ whenever $X_0=x\in \S\setminus S_0$. 
\end{theorem}

The monotonicity assumptions \textbf{D3-D4} are definitely not necessary for persistence (see, e.g., \citet[Theorem 2]{tpb-09}) . Hence, it is natural to conjecture that 
\begin{conjecture}
Assume \textbf{D1-D2} hold and the assumptions of Theorem~\ref{thm:bounded} hold. If $\gamma>0$, then \eqref{eq:structured} is almost surely persistent and persistent in probability. 
\end{conjecture}

To illustrate the applicability of Theorem~\ref{thm:persist}, we consider an example from \citet{tpb-09}.

\begin{example}[Biennial plants]{\rm

Biennial plants typically flower only in the  second year of their existence after which they die. However, for many biennial species, individual plants may exhibit  delayed flowering in which they flower in a later year. Delayed flowering can serve as a bet-hedging strategy in uncertain environments~\citep{roerdink-87} provided a detailed analysis of a density-independent model of delayed flowering. Here, I describe a density-dependent version of his model. Let $X_t^1$ denote the abundance of one year old individuals in year $t$ and $X_t^2$ denote the abundance of individuals greater than one year old in year $t$. Let $p$ be the probability that a plant flowers during its second year.  Let $\xi_{t+1} s_1(X_t)$ be the mean number of germinating seeds produced by a flowering plant in year $t$ where $s_1(x_1,x_2)=\frac{1}{1+b_1(x_1+x_2)}$.  Let $s_2(x_1,x_2)=\frac{a}{1+b_2(x_1+x_2)}$ be the probability that a plant survives to the next year. Then  the plant dynamics are given by 
\begin{equation}\label{eq:biennial}
X_{t+1}=\begin{pmatrix}0& p\,\xi_{t+1} s_1(X_t) \\ s_2(X_t) & (1-p) s_2(X_t)   \end{pmatrix}X_t
\end{equation}
Let $\xi_1,\xi_2,\dots$ be a sequence of independent random variables that are Gamma distributed i.e. having the probability density function
\[
g(t)=\frac{1}{\theta^b\Gamma(k)}t^{k-1} \exp(-t/\theta),
\]
where the scale parameter is $\theta>0$,  the shape parameter is $k>0$ and $
\Gamma(k)=\int_0^\infty t^{k-1}e^{-t}\,dt$.
The mean and variance of $\xi_1$ are given by $k\theta$ and $k\theta^2$. 
\citet{roerdink-87} found an explicit formula for the dominant Lyapunov exponent. For $0<p<1$, the dominant Lyapunov exponent is given by 
\[
\gamma = \ln{a(1 -p)} + K^{-1}\int_0^\infty  \ln(1 + t)\,t^{k-1}(1 + t)^{-k} e^{-zt}\,dt
\]
where $K=\int_0^\infty  t^{k-1}(1+t)^{-k}e^{-zt}\,dt$ and $z=(1-p)^2/(\theta \,p)$. For $p=0$, $\gamma=\ln a$, while for $p=1$, $\gamma=\frac{1}{2}\left( \ln(a \theta)+\psi(a)\right)$ where $\psi(a)$ is the digamma function.  Roerdink proved that $\frac{\partial \gamma}{\partial p}$ is positive at $p=0$ and approaches $-\infty$ as $p$ approaches $1$. Hence, the stochastic growth rate $\gamma$ of the population is maximized by the population playing an appropriate  bet hedging strategy for flowering (i.e. $p$ strictly between $0$ and $1$).  Therefore, Theorem~\ref{thm:persist} implies that  persistence is more likely for populations playing a bet hedging strategy.  }
\end{example}

\section{Concluding remarks and future directions}
 
 While these results represent the promising beginnings of a general theory for persistence of stochastic difference equations, there is still much work to be done. The conjectures sprinkled throughout this review are merely the tip of the iceberg. To give a sense of some other issues, I discuss three directions for future research. 

 First and foremost, one can ask ``Is there a general theorem unifying all the particular cases of the persistence results?'' For dissipative  deterministic models, there are two characterizations of  uniform persistence. The average Lyapunov characterization due to \citet{hutson-84,hutson-88} requires the existence of a non-negative function that increases on average for trajectories near the extinction set. Alternatively, \citet{butler-waltman-86}, \citet{garay-89} and \citet{hofbauer-so-89} provided topological characterizations of uniform persistence in terms of Morse decompositions and stable sets. Intuition suggests there should be a stochastic analog of the average Lyapunov function characterization. Indeed \citet{benaim-etal-08} used average Lyapunov functions for the deterministic models to prove persistence for the stochastically perturbed models. Whether this argument can be extended is an exciting and challenging open problem.

Many environmental signals are positively autocorrelated in time~\citep{vasseur-yodzis-04}. Understanding the impacts of these autocorrelations for population persistence and species interactions is a very active area of research in theoretical and empirical population biology~\citep{heino-etal-00,gonzalez-holt-02,roy-etal-05,schwager-etal-06,matthews-gonzalez-07,reuman-etal-08,roy-holt-09,prsb-10}. To account for these autocorrelations,  the environmental sequence of random variables, $\xi_1,\xi_2,\dots$, can no longer can be independent. However, they may be stationary or even asymptotically stationary i.e. $|\P[X_t\in A_0,\dots, X_{t+n}\in A_n]-\P[X_{t+s}\in A_0,\dots, X_{t+s+n}\in A_n]|\to 0$ as $t\to\infty$ for any Borel sets $A_0,\dots, A_n\subset \Omega$, $n\ge 1$ and $s\ge 0$. \citet[Theorem 4]{tpb-09} made this extension for structured populations models satisfying monotonicity assumptions as in Theorem~\ref{thm:persist}. Similar extensions have yet to be made for multispecies models or non-monotonic structured models. 

In order to apply the methods reviewed here, there is a desperate need for general methods to estimate the mean per-capita growth rates $\lambda_i(\mu)$ and the dominant Lyapunov exponent $\gamma$. One approach is to consider ``small noise approximations'' of  these quantities when distribution of $\xi_t$ is close to a Dirac measure. For structured population models where $A(0,\xi_t)=A+\varepsilon B_t$ for a fixed non-negative matrix $A$ and  a sequence of random matrices $B_t$ with mean $0$, \citet{tuljapurkar-90} developed second and higher order approximations of the dominant Lyapunov exponent $\gamma$ with respect to $\varepsilon$.  This approximation  yielded many useful insights into stochastic demography and metapopulation persistence~\citep{wiener-tuljapurkar-94,boyce-etal-06,tuljapurkar-haridas-06,morris-etal-08,tuljapurkar-etal-09,prsb-10}. Alternatively, for models of competing species, Chesson  developed methods for estimating $\lambda_i(\mu)$ when the deterministic dynamics converge to an equilibrium and the small noise generates small demographic fluctuations around this equilibrium~\citep{chesson-88,chesson-94,chesson-00}. Extending these methods to arbitrary species interactions and random perturbations from non-equilibrium dynamics, however, is an important remaining challenge. 

 In conclusion, this review highlights the progress, challenges, and opportunities in using stochastic difference equations to understand the conditions necessary for population persistence. The speed at which this review becomes hopelessly outdated may be the best measurement of its success.  \\

\noindent \textbf{Acknowledgments.} The author thanks Peter Ralph and two anonymous reviewers for comments on the manuscript. This research was supported by National Science Foundation Grants DMS-0517987 and DMS-1022639. 

\bibliography{../../seb}

\end{document}